\documentclass[12pt]{article}
\usepackage{amsfonts, amsmath, amssymb,latexsym,amsthm}
\usepackage[margin=1.1in]{geometry}

\def\max{{\mathfrak m}}                   
\def\res{{\bf k}}                   
\newcommand{\m}{\mathfrak m}
\newcommand{\ud}{{\underline{d}}}
\newcommand{\uf}{{\underline{f}}}
\newcommand{\uF}{{\underline{F}}}
\def\ann{\operatorname{Ann_R}}
\def\soc{\operatorname{soc}}

\def\HF{\operatorname{HF}}
\def\Res{\operatorname{Res}}
\def\det{\operatorname{det}}
\def\Jac{\operatorname{Jac}}

\def\spec{{\bf Spec}}
\def\HF{\operatorname{HF}}

\def\deg{\operatorname{deg}}
\def\dim{\operatorname{dim}}

\newcommand{\proj}[1]{\mathbb P_{\res}^{#1}}

\newtheorem{theorem}{Theorem}[section]
\newtheorem{lemma}[theorem]{Lemma}

\newtheorem{proposition}[theorem]{Proposition}
\newtheorem{remark0}[theorem]{Remark}
\newtheorem{example0}[theorem]{Example}

\newenvironment{example}{\begin{example0}\rm}{\end{example0}}
\newenvironment{remark}{\begin{remark0}\rm}{\end{remark0}}

\newcommand{\propref}[1]{Proposition~\ref{#1}}

\newcommand{\lemref}[1]{Lemma~\ref{#1}}

\title{  \bf  Towards a characterization of the inverse systems of complete intersections.
\footnote{ 2010 {\it Mathematics Subject Classification}.
Primary 13H10; Secondary 13H15; 14C05
\newline
\indent \ \ {\it Key words and Phrases:} Number of generators, Artin Gorenstein local rings, Inverse system.}}
\author{\large   J. Elias
\thanks{Partially supported by PID2022-137283NB-C22.}
}

\date{ \today}

\begin{document}
\maketitle
\begin{abstract}
In this paper we give conditions on a homogeneous polynomial for which the associated graded Artin algebra is a complete intersection.
\end{abstract}

\section{Introduction}

We denote by $S=\res[x_0,\dots,x_n]$  the polynomial ring on the  variables $x_0,\dots,x_n$ over an algebraically closed field $\res$ of characteristic zero.
We denote  the homogeneous maximal ideal of $S$  by $\max=(x_0,\dots,x_n)$.

Macaulay's duality, instance of Matlis' duality, establish a reversing one-to-one correspondence between the Artin Gorenstein quotients $A=S/I$ and their inverse system $I^\perp=\langle H\rangle$, see \propref{mac}.
In this paper we address the following problem:
characterize the polynomials $H$ such that its Macaulay's dual $I=\ann\langle H\rangle$ is a complete intersection, see \cite[Chapter 9, L]{IK99}.
Very few instances of this problem  are known. For example, if $H$ is a monomial then $\ann\langle H\rangle$ is a monomial complete intersection.
In \cite{HWW17} the case of quadratic complete intersections is considered; see \cite{HWW17} and its reference list for some other known cases.

In this paper we give explicit conditions on a homogeneous polynomial $H$ for which $I=\ann\langle H \rangle$ is a complete intersection, \propref{char}.
The key idea of this paper is to consider the resultant of a family of homogeneous forms
$\uf=(f_0,\dots ,d_n)$ and the determinant of its jacobian matrix.
Recall that the resultant of a family of homogeneous forms is the generator of a closed projective hypersurface, \cite{EM07}.
In the case considered in this paper, we get a locally constructible set, \propref{first-set}.

In the second section we characterize   the polynomials $H$ such that its Macaulay's dual $I=\ann\langle H\rangle$ is a complete intersection, \propref{char}.
In the third section we apply this result to short algebras.

\medskip
\noindent
{\sc{Preliminaries.}}
Next,\textbf{} we fix some notations and we recall the basic results on Macaulay inverse systems,
see  \cite{IK99} for more results on inverse systems.

Let $\Gamma=\res[y_0,\dots,y_n]$ be the polynomial ring on the  variables $y_0,\dots,y_n$.
This  polynomial ring can be considered a $S$-module  by derivation that we denote by $\circ$, defined by
$$
\begin{array}{ cccc}
\circ: & S  \times \Gamma  &\longrightarrow &  \Gamma   \\
                       &       (f , g) & \to  &  f  \circ g = f (\partial_{y_1}, \dots, \partial_{y_n})(g)
\end{array}
$$

\noindent
where $  \partial_{y_i} $ denotes the partial derivative with respect to $y_i$;
$S$ is the injective hull of the residue field $\res$ as $S$-module, \cite{Gab59}.

Given a  polynomial $H\in \Gamma$ we denote by
$\langle H \rangle$ the sub-$S$-module of $\Gamma$ generated by
$H$.
Notice that $\langle H \rangle$  is the $\res$-vector subspace of $\Gamma$ generated
by  $H$ and its derivatives of any order.
On the other hand, we denote by $\langle H \rangle_\res$ the one-dimensional $\res$-vector space generated by $H$.

\begin{proposition}[Macaulay-Matlis duality, \cite{IK99}]
\label{mac}
There is a order-reversing bijection $\perp$ between the set of finitely generated (graded) sub-$S$-submodules of $\Gamma$ and the set of (homogeneous) $\m$-primary ideals of $S$ given by:
if $M$ is a  submodule of $\Gamma$ then $M^\perp=(0:_S M)\subset S$, and $I^\perp=(0:_\Gamma I)\subset \Gamma$ for an  ideal $I\subset S$.
An ideal $I$ is Gorenstein if and only in $I^\perp$ is generated by a single polynomial $F\in \Gamma$.
 \end{proposition}

\noindent
$I^\perp$ is the so-called Macaulay's inverse system of $I$.

Given an Artinian graded quotient $A=S/I$
we denote  by $\soc(A)=(0:_A \max)$ the socle of $A$.
The ring $A$ is Gorenstein if and only if $\dim_\res(\soc(A))=1$, if this is the case then  $\soc(A)$ is a principal ideal of $A$ generated by a homogeneous element of degree $s$. This integer is the socle degree of $A$.
By Macaulay's duality there exists a homogeneous element $F$ of degree $s$ such that $I^\perp=\langle F\rangle$.

The examples of this paper are done by using the  Singular library  \cite{E-InvSyst14},
\cite{DGPS}.

\section{Characterization of complete intersections}

The following result is a consequence of \cite{Ku92}, but in the Artinian case  we can give an easier proof by using Macaulay's duality.

\begin{lemma}
\label{KU}
Let $A=S/K$ be a graded Gorenstein ring of socle degree $s$, and let $J$ be a homogeneous ideal such that $K\subset J$.
If  $\soc(R/J)_s\neq 0$ then $K=J$.
\end{lemma}
\begin{proof}
We have $J^\perp\subset K^\perp=\langle F \rangle$
with $F$ a homogeneous form of degree $s$.
Since $\soc(R/J)_s\neq 0$ there exist a homogeneous polynomial  $G\in J^\perp\setminus \{0\}$ with $\deg(G)=s$.
Hence $0\neq G\in K^\perp_s=\langle F\rangle_\res$.
Since $\dim_\res\langle F\rangle_\res =1$ there exist
$\lambda\in \res\setminus\{0\}$ such that
$G=\lambda F$.
From this we get that
$$
K^\perp=\langle F\rangle=\langle G\rangle\subset J^\perp.
$$
Hence $K^\perp=J^\perp$ and by Macaulay's duality $K=J$.
\end{proof}

 We say that a homogeneous ideal $I\subset S$ is a complete intersection of type  $\ud=(d_0,\dots,d_n)\in \mathbb N^{d+1}$, with $2\le d_0\le\dots\le d_n$,
if there exists homogeneous regular sequence  $\uf=\{f_0,\dots ,f_n\}$ minimally generating $I$ such that $\deg(f_i)=d_i$, $i=0,\dots,n$.

Given an integer $d\ge 1$ we denote by $\proj{N(d)}$ the projective space parametrizating the forms of degree $d$ on $x_0,\dots,x_n$, where $N(d)=\binom{n+d}{n}-1$.
Given integers $2\le d_0\le\dots \le d_n$ we consider the multi-projective space
$
{\mathbb V}= \prod_{i=0}^n \proj{N(d_i)}
$
parameterizing the $(n+1)$-pla of homogeneous forms $\uF=\{F_0,\dots,F_n\}$,  where $F_i$  is a homogeneous form of degree $d_i$, $i=0,\dots, n$.
Recall that the resultant $\Res(\uF)$ defines an irreducible  hypersurface of $\mathbb V$,
\cite[Chapitre 5, Section 5.3]{EM07}.
In fact, $\Res(\uF)$ is the generator of the defining ideal of the variety of $\mathbb V$ obtained by projection on the first component of the incidence variety
$\{(\uF, p)\in {\mathbb V} \times {\mathbb P}_\res^n \mid F_i(p)=0, i=0,\dots, n\}$.
The resultant $\Res(\uF)$ is a polynomial of  degree $\prod_{i\neq j} d_j$ on the coefficients of $F_i$,
$i=0,\dots, n$, \cite[Proposition 5.21]{EM07}.
The resultant can be computed effectively,
\cite[Algorithme 5.22]{EM07}, and can be a dense polynomial, see \cite[Example 5.20]{EM07}.

Given a $n+1$-pla  $\uf=(f_0,\dots, f_n)$ of homogeneous elements of $S$ of degree $d_0,\dots, d_n$ the ideal $(\uf)$ is $\max$-primary if and only if $\Res(\uf)\neq 0$, \cite[Proposition 5.17]{EM07}.
These conditions are equivalent to $\dim_\res(S/(\uf))< \infty$, or that $\uf$ is a maximal regular sequence of $S$.

The main result of this paper is the following:

\begin{proposition}
\label{char}
Let $H\in S\setminus\{0\}$ be a homogeneous polynomial.
Then the ideal $\ann\langle H\rangle$ is a complete intersection of type
$\ud=(d_0,\dots,d_n)$ if and only if there exists homogeneous forms
$\uf=\{f_0,\dots ,f_n\}\subset S$ such that $\deg(f_i)=d_i$, $i=0,\dots,n$, and
\begin{enumerate}
\item[(i)] $\Res(\uf)\neq 0$,
\item[(ii)] $\uf\circ H=0$, and
\item[(iii)]  $\det(\Jac(\uf))\circ H\neq 0$.
\end{enumerate}
\noindent
If these equivalent conditions hold then
$\ann\langle H\rangle=(f_0,\dots, f_n)$ and
$\deg(H)=\sum_{i=0}^n d_i -(n+1).$
\end{proposition}
\begin{proof}
Assume that $K=\ann\langle H\rangle$ is a complete intersection ideal generated by a homogeneous regular sequence $\uf=\{f_0,\dots ,f_n\}\in S$ of type $(d_0,\dots,d_n)$.
Since $(f_0,\dots ,f_n)\subset K$, by Macaulay duality we get $(ii)$.
Being $\{f_0,\dots ,f_n\}$ a regular sequence we have $\Res(\uf)\neq 0$, \cite{SS01} or \cite[Proposition 5.17]{EM07}.
By hypothesis $K$ is a complete intersection, so from \cite{SS75}, see also \cite[Theorem 2.4]{Vas91}, we have that the coset of $\det(\Jac(\uf))$
in $S/K$ is a generator of its socle, in particular $\det(\Jac(\uf))\circ H\neq 0$.

Let now $\uf=\{f_1,\dots ,f_n\}\in S$ be a sequence of homogeneous elements of $S$ satisfying the conditions $(i)$, $(ii)$ and $(iii)$.
From $(i)$ we get that  $\uf$ is a regular sequence, \cite{SS01} or \cite[Proposition 5.17]{EM07}.
From $(ii)$ and Macaulay's duality we get that $K=(f_0,\dots,f_n)\subset J=\ann\langle H\rangle$, so we can consider the natural projection
\begin{equation}
\label{eq1}
\frac{S}{K} \longrightarrow \frac{S}{J}
\end{equation}

Being $K$  a complete intersection of height $n+1$ the socle of $S/K$ is concentrated in a single degree
$s=\sum_{i=0}^n d_i -(n+1)$, and it is generated by the coset of $\det(\Jac(\uf))$,
\cite{SS75}, see also \cite[Theorem 2.4]{Vas91}.
By \eqref{eq1} the coset of $\det(\Jac(\uf))$ in $S/J$  belongs to its socle.

From $(iii)$ we get that the ideal  of $S/J$ generated by  $\det(\Jac(\uf))$ is not zero and of degree $s$.
Since
 the coset of $\det(\Jac(\uf))$ in $S/J$ is non-zero and  belongs to $\soc(R/J)$,
from \lemref{KU} we deduce that $K=J$, i.e. $(f_0,\dots,f_n)=\ann\langle H\rangle$.

We know that the coset of  $\det(\Jac(\uf))$ in $S/K$ is a generator of its socle, so
$\deg(H)=\deg(\det(\Jac(\uf)))=\sum_{i=0}^n \deg(f_i) -(n+1)$.
\end{proof}

\medskip
\begin{remark}
Let $H\neq 0$ be a homogeneous polynomial.
Since $\ann\langle H \rangle$ is of maximal height
there always exists a sequence $\uf$ satisfying $(i)$ and $(ii)$.
\end{remark}

Given integers $1\le d_0\le\dots \le d_n$  we consider the multi-projective space
$
{\mathbb E}= \mathbb{V}\times \proj{N(s)}
$
parameterizing the pairs $(\uf,H)$ where $\uf=\{f_0,\dots, f_n\}$ is a $(n+1)$-pla of homogeneous forms of degrees $d_1,\dots, d_n$, and $H$ is a homogeneous polynomial on $y_0,\dots, y_n$ of degree $s=\sum_{i=0}^n d_i -(n+1)$.
We denote by $\pi$ the natural projection in the second component of $\mathbb E$:
$
\pi:{\mathbb E}\longrightarrow \proj{N(s)}.
$

Let $U$ be the open subset of $\mathbb{E}$ defined by the $(n+1)$-plas $\uf$ such that
$\Res(\uf)\neq 0$.
We consider  the locally closed subset $Z_1$ (resp. closed subset  $Z_2$) of $\mathbb E$ defined by the condition
$\uf \circ H=0$ and $\uf\in U$ (resp. $\det(\Jac(\uf))\circ H=0$).
For all $[H]\in \proj{N(s)}$, we denote the fiber of $[H]$ in $Z_i$ by $\pi$, $i=1,2$,
$$(Z_i)_H=Z_i\cap \pi^{-1}([H])=\{(\uf,[H])\in Z_i\}.$$

\medskip
\begin{proposition}
\label{first-set}
Let $H$ be a degree $s$ homogeneous form.
Then $\ann\langle H\rangle$ is  a complete intersection of type $(d_0,\dots ,d_n)$ if and only if
$$
(Z_1)_H \nsubseteq (Z_2)_H.
$$
The set of homogeneous polynomials $[H]\in \proj{N(s)}$ such that $(Z_1)_H \nsubseteq (Z_2)_H$
is a locally constructible set.
\end{proposition}
\begin{proof}
Assume that $\ann\langle H\rangle$ is not a complete intersection.
Let $(\uf,H)$ be an element of $(Z_1)_H$.
From \propref{char} we get  $\det(\Jac(\uf))\circ H=0$, so  $(\uf,H)\in (Z_2)_H$.

Let $H$ be an homogeneous form of degree $s$ such that
$(Z_1)_H \subset (Z_2)_H$.
Assume that
$\ann\langle H\rangle$ is a complete intersection of type $d_0,\dots ,d_n$.
Then there exists a family $\uf=(f_0,\dots,f_n)$ satisfying the conditions of
\propref{char}.
In particular, condition $(i)$ and $(ii)$ is equivalent to $(\uf,H)\in Z_1$.
By hypothesis $(\uf,H)\in Z_2$ so $\det(\Jac(\uf))=0$.
This a contradiction with the condition $(iii)$ of \propref{char}.

From \cite[Corollaire 9.5.2]{EGA-IV-III} we get that
the set of homogeneous polynomials $[H]\in \proj{N(s)}$ such that $(Z_1)_H \nsubseteq (Z_2)_H$
is a locally constructible set.
\end{proof}

The computation of the defining equations of the locally constructible  set of the last result, in contrast the case of the resultant, could be difficult.
In the next section we will study a concrete example where \propref{char} is applied for proving that some Gorenstein ideal $\ann\langle H\rangle$ is
a complete intersection.
First we review the example 3.5 of \cite{HWW17}.

\begin{example}
In this example we follow the notations of \cite[Example 3.5]{HWW17}.
We consider the ring of polynomials $S=\res[v, w,x,y,z]$.

\noindent
$(1)$ of \cite[Example 3.5]{HWW17}.
Given $H=vwxyz+wxyz^2$,  $\ann\langle H\rangle$ contains the sequence $\uf=\{v^2, w^2,x^2, y^2 z^2 -2vz\} $, i.e. satisfies the condition $(ii)$ of \propref{char}.
Since $S/(\uf)$ is Artininan we get condition $(i)$ of \propref{char}.
On the other hand, we get $\det(\Jac(\uf))=2^5 v w x y (z-v)$, so
$\det(\Jac(\uf))\circ H=2^5\neq 0$.
Hence $\ann\langle H\rangle=(\uf)$ is a complete intersection.

\noindent
$(2)$ of \cite[Example 3.5]{HWW17}.
We proceed as the previous case.

\noindent
$(3)$ of \cite[Example 3.5]{HWW17}.
Let us consider $H=vwxyz+yz^4$.
If $\ann\langle H\rangle$ is a complete intersection, since $5=d_0+\dots +d_4-5$, then  we get that $d_0=\dots =d_4=2$.
The $\res$-vector space of forms of degree two of $\ann\langle H\rangle$ admits as $\res$-base the monomials: $v^2, w^2, x^2, y^2$, i.e. a $\res$-base of  $\uf$ satisfying $(ii)$ of \propref{char}.
Hence,  any $5$-pla  $\uf=(f_0,\dots, f_5)$  formed by $5$ forms of degree two are linearly dependent, so $\det(\Jac(\uf))=0$ and then $\ann\langle H\rangle$ is not a complete intersection.
\end{example}

\section{Case study: short algebras}

In this section we consider  short Gorenstein algebras: local Gorenstein algebras $A=\res[\![x_1,\dots,x_r]\!]/I$ with Hilbert function $\{1,r,r,1\}$.
Recall that in the main result of \cite{ER12} we prove that these algebras are isomorphic
to  their associated graded ring, i.e. we can consider that $A\cong \res[x_1,\dots,x_r]/K$ with $K$ a homogeneous ideal.
On the other hand if $K$ is a complete intersection then $r\le 3$.
The cases $r\le 2$ are computed directly, see \cite{ER12}.
In this example we study the case $r=3$, following the general notations of this paper we have that  $n=2$.
Hence  $K=\ann\langle H\rangle$ is an ideal of $S=\res[x_0,x_1,x_2]$ with $H$ a degree $3$ homogeneous polynomial.

In \cite{ER12} we computed the isomorphism classes of these algebras in terms of the moduli of the degree three projective plane curve defined by $H$.
Recall that a plane elliptic  cubic curve $C\subset \mathbb P^{2}_\res$ is defined, in a suitable
system of coordinates, by a   Legendre's equation
$$
H_\lambda=y_1^2 y_2- y_0(y_0-y_2)(y_0-\lambda y_2)
$$
with $\lambda \neq 0,1$.
The  $j$-invariant of $C$ is:
$$
j(\lambda)= 2^8 \; \frac{(\lambda^2-\lambda+1)^3}{\lambda^2 (\lambda -1 )^2}.
$$
It is well known that
two plane elliptic  cubic  curves $C_i=V(H_{\lambda_i})\subset \mathbb P^{2}_\res=\spec(S)$, $i=1,2$, are projectively isomorphic
if and only if $j(\lambda_1)=j(\lambda_2)$.
In \cite[Proposition 3.7]{ER12} we proved that:

\begin{proposition}\cite[Proposition 3.7]{ER12}
\label{ell}
Let $A$ be an Artinian   Gorenstein local $\res$-algebra   with Hilbert function
$\HF_A=\{1,3,3,1\}$. Then $A$ is isomorphic to one and only one of the following graded
  quotients of  $S=\res[x_0,x_1,x_2]$:
$$
\begin{array}{|c|c|c|}           \hline
 \text{ Model } A=S/K& \text{Inverse system }& \text{Geometry of } C=V(K)\subset \mathbb P^{2}_\res \\ \hline
 (x_0^2,x_1^2,x_2^2) & y_0 y_1 y_2 & \text{Three independent lines} \\  \hline
 (x_0^2,x_0x_3^2,x_2x_1^2,x_1^3,x_2^2+x_0x_1) & y_1(y_0y_1- y_2^2) & \text{Conic and a tangent line} \\  \hline
 (x_0^2,x_1^2,x_2^2+6x_0x_1) &  y_3(y_1y_2- y_3^2)& \text{Conic and a non-tangent line} \\  \hline
 (x_1^2, x_0x_1,x_0^2+x_1^2-3 x_0x_2) &  y_1^2 y_2- y_0^2( y_0+ y_2)& \text{Irreducible nodal cubic} \\  \hline
 (x_2^2,x_0x_1, x_0x_2,x_1^3,x_0^3+ 3x_1^2x_2) & y_1^2 y_2 - y_0^3 & \text{Irreducible cuspidal cubic} \\  \hline
 (x_1x_2,x_0x_2,x_0x_1,x_1^3-x_2^3,x_0^3-x_2^3) & y_0^3+y_1^3+y_2^3  &  \text{Elliptic Fermat curve} \\ \hline
 K_\lambda=(x_0x_1,F_1,F_2) & H_\lambda,\; j(\lambda)\neq 0& \text{Elliptic non Fermat curve} \\ \hline
 \end{array}
 $$
 with
$F_1=\lambda^2 x_0^2 + \lambda (1+\lambda) x_0 x_2 + (\lambda^2-\lambda+1)x_2^2$,
$F_2=\lambda^2 x_1^2 + \lambda x_0 x_2 + (1+\lambda)x_2^2$,
 and $K_{\lambda_1}\cong  K_{\lambda_2}$ if and only if $j(\lambda_1)=j(\lambda_2)$.
\end{proposition}

Notice that the minimal number of generators of the  height three Gorenstein ideal $K$ of the last result are odd, as the main theorem of \cite{Wat73} prescribes.
All  cases in \propref{ell} can be easily computed  by using Singular but the last one, \cite{E-InvSyst14}.
In this  case a parameter is involved: the $j$-invariant of the elliptic curve.
In \cite{ER12}  we settle  this case by showing first that
$J=(x_0x_1,F_1,F_2)\subset K_\lambda$ and then
proving that $\dim_\res(S/J)=8$.
We can prove this fact by using \propref{char}.
We know that $J=(x_0x_1,F_1,F_2)\subset K_\lambda=\ann\langle H_\lambda \rangle$, so $\uf=\{x_0x_1,F_1,F_2\}$ satisfies condition $(ii)$ of \propref{char}.
On the other hand it is easy to prove that
$J$ is $\max$-primary, so we have condition $(i)$ of \propref{char}.
Since $j(\lambda)\neq 0$
$$
\det(\Jac(x_0x_1,F_1,F_2))\circ H_\lambda= -16\lambda^2 (\lambda^2-\lambda+1)\neq 0.
$$
Hence we get  condition $(iii)$ of \propref{char} and then $(x_0x_1,F_1,F_2)=K_\lambda$.


\medskip
\noindent
Joan Elias\\
Departament de Matemàtiques i Informàtica\\
Universitat de Barcelona\\
Gran Via 585, 08007 Barcelona, Spain\\
e-mail: {\tt elias@ub.edu}
\end{document}